\documentclass{amsart}
\usepackage{geometry}                
\usepackage{graphicx}
\usepackage{amssymb}
\usepackage{epstopdf,amsmath}

\newtheorem{prop}{Proposition}[section]

\newtheorem{cor}[prop]{Corollary}

\newtheorem{theorem}[prop]{Theorem}
\newtheorem{conj}[prop]{Conjecture}
\newtheorem{question}[prop]{Question}

\DeclareMathOperator{\Deg}{Deg}
\DeclareMathOperator{\Flec}{Flec}

\newcommand{\FF}{\mathbb{F}}
\newcommand{\CC}{\mathbb{C}}

\newcommand{\RR}{\mathbb{R}}

\title{Ruled surface theory and incidence geometry}

\author{Larry Guth}

\begin{document}

\begin{abstract} We survey the applications of ruled surface theory in incidence geometry.  We discuss some of the proofs and raise some open questions.
\end{abstract}

\maketitle

\section{Introduction}

In the last five years, there have been some interesting applications of ruled surface theory in incidence geometry, which started in the work that Nets Katz and I did on the Erd{\H o}s distinct distance problem \cite{GK}.  In this essay, we survey the role of ruled surface theory in incidence geometry.  

Ruled surface theory is a subfield of algebraic geometry.  A ruled surface is an algebraic variety that contains a line through every point.  Ruled surface theory tries to classify ruled surfaces and to describe their structure.  The incidence geometry questions that we study here are about finite sets of lines.  A ruled surface can be roughly thought of as an algebraic family of lines.  Some of the questions in the two fields are actually parallel, but they take place in two different settings -- the discrete setting and the algebraic setting.  We will discuss a connection between these two settings.  

The applications of ruled surface theory are the most technical part of \cite{GK}.  I wrote a book about polynomial methods in combinatorics, \cite{G}, including a chapter about applications of ruled surface theory.  My goal in that chapter was to give a self-contained proof of the results from \cite{GK} and to make the technical details as clean as I could.  In this essay, my goal is to give an overview -- we will discuss some results, some of the main ideas in the proofs, and some open problems.  

Here is an outline of the survey.  In Section 2, we discuss the combinatorial results that have been proven using ruled surface theory.  In Section 3, we sketch a proof of the simplest result in Section 2.  In the course of this sketch, we try to explain some tools from ruled surface theory and how those tools help us to understand combinatorial problems.  In Section 4, we discuss some open problems, exploring what other things we could hope to learn about incidence geometry by using the theory of ruled surfaces.  

I would like to thank the anonymous referees for helpful suggestions.

\section{Results and open questions}

In \cite{GK}, ruled surface theory is used to prove an estimate about the incidence geometry of lines in $\RR^3$ (and this estimate eventually leads to estimates about the distinct distance problem).  Recall that if $\frak L$ is a set of lines, then a point $x$ is an $r$-rich point of $\frak L$ if $x$ lies in at least $r$ lines of $\frak L$.  We write $P_r(\frak L)$ for the set of $r$-rich points of $\frak L$.  The theorem says that a set of lines in $\RR^3$ with many 2-rich points must have some special structure.

Before stating the theorem, we do a couple examples.  Because any two lines intersect in at most one point, a set of $L$ lines can have at most $L \choose 2$ 2-rich points.  A generic set of lines in the plane achieves this bound.  So a set of $L$ lines in $\RR^3$ can have at most $L \choose 2$ 2-rich points, and there is an example achieving this bound where all the lines lie in a plane.  This suggests the following question: if a set of $L$ lines in $\RR^3$ has on the order of $L^2$ 2-rich points, does it have to be the case that many of the lines lie in a plane?  Interestingly, the answer is no.  The counterexample is based on a degree 2 algebraic surface.  Consider the surface defined by the equation

$$ z - xy = 0. $$

\noindent This surface contains many lines.  For any $a \in \RR$, the surface contains the line parametrized by

$$ t \mapsto (a, t, at). $$

\noindent Similarly, for any $b \in \RR$, the surface contains the line parametrized by

$$ t \mapsto (t, b, tb). $$

\noindent If we choose $L/2$ values of $a$ and $L/2$ values of $b$, we get a set of $L$ lines contained in our surface with $L^2/4$ 2-rich points.  Any plane contains at most 2 of these lines.  The polynomial $z-xy$ is not unique: there are many other degree 2 polynomials that work equally well.

But in some sense, this is the only counterexample.  If a set of $L$ lines in $\RR^3$ has many 2-rich points, then it must be the case that many of the lines lie in either a plane or a degree 2 surface.  Here is a precise version of this statement.

\begin{theorem} \label{2richR3} (Katz and G., \cite{GK}) There is a constant $K$ so that the following holds.  Suppose that $\frak L$ is a set of $L$ lines in $\RR^3$.  Then either

\begin{itemize}

\item $|P_2(\frak L)| \le K L^{3/2}$ or

\item there is a plane or degree 2 algebraic surface that contains at least $L^{1/2}$ lines of $\frak L$.  

\end{itemize}

\end{theorem}

By using this theorem repeatedly, we can prove a stronger estimate, which roughly says that if $|P_2(\frak L)|$ is much bigger than $L^{3/2}$, then almost all of the 2-rich points ``come from'' planes or degree 2 surfaces.  

\begin{cor} Suppose that $\frak L$ is a set of $L$ lines in $\RR^3$.  Then, there are disjoint subsets $\frak L_i \subset \frak L$ so that

\begin{itemize}

\item For each $i$, the lines of $\frak L_i$ lie in a plane or a degree 2 surface.

\item $ | P_2(\frak L) \setminus \cup_i P_2(\frak L_i ) | \le K L^{3/2}$. 

\end{itemize}

\end{cor}

\begin{proof} We prove the corollary by induction on the number of lines.  If $|P_2(\frak L)| \le K L^{3/2}$, then we are done.  Otherwise, by Theorem \ref{2richR3}, there is a subset $\frak L_1 \subset \frak L$, so that $| \frak L_1 | \ge L^{1/2}$ and all the lines of $\frak L_1$ lie in a plane or degree 2 surface.  We let $\frak L' = \frak L \setminus \frak L_1$.  By induction, we can assume the corollary holds for $\frak L'$ -- giving us disjoint subsets $\frak L_i \subset \frak L'$.  Suppose that a point $x$ is in $P_2(\frak L) \setminus \cup_i P_2(\frak L_i)$.  Then 
either $x$ lies in a line from $\frak L_1$ and a line from $\frak L'$, or else $x \in P_2(\frak L') \setminus \cup_i P_2(\frak L_i)$.  The lines of $\frak L_1$ all lie in a plane or regulus, and each line of $\frak L'$ intersects this plane or regulus at most twice, so the number of points of the first type is at most $2 L$.  By induction, the number of points of the second type is at most $K |\frak L'|^{3/2} \le K (L - L^{1/2})^{3/2}$.  In total, we see that

$$ | P_2(\frak L) \setminus \cup_i P_2(\frak L_i ) | \le 2 L + K (L - L^{1/2})^{3/2} \le K L^{3/2}, $$

\noindent closing the induction.  (In the last step, we have to assume that $K$ is sufficiently large, say $K \ge 100$.)

\end{proof}

Ruled surface theory plays a crucial role in the proof of Theorem \ref{2richR3}.  We will explain how in the next section.  Before doing that, we survey generalizations of Theorem \ref{2richR3}, discussing both known results and open problems.

The first question we explore is the choice of the field $\RR$.  Does the same result hold over other fields?  This question was answered by Kollar in \cite{K}.  He first proved that Theorem \ref{2richR3} holds over any field of characteristic zero.  Next he addressed fields of finite characteristic.  As stated, Theorem \ref{2richR3} does not hold over finite fields.  There is a counterexample over the field $\FF_q$ when $q$ is not prime -- see \cite{EH} for a description of this example.  Nevertheless, Kollar proved that Theorem \ref{2richR3} does hold over fields of finite characteristic if we add a condition on the number of lines.

\begin{theorem} \label{2richk3} (Corollary 40 in \cite{K}) Suppose that $k$ is any field.  Suppose that $\frak L$ is a set of $L$ lines in $k^3$.  If the characteristic of $k$ is $p > 0$, then assume in addition that $L \le p^2$.  Then either

\begin{itemize}

\item $|P_2(\frak L)| \le K L^{3/2}$ or

\item there is a plane or degree 2 algebraic surface that contains at least $L^{1/2}$ lines of $\frak L$.  

\end{itemize}

\end{theorem}

(In particular, this implies that Theorem \ref{2richR3} holds over prime finite fields $\FF_p$.  The reason is that $|P_2(\frak L)| \le | \FF_p^3 | = p^3$.  So if $L \ge p^2$, then $|P_2(\frak L)| \le L^{3/2}$ trivially, and if $L \le p^2$, then Theorem \ref{2richk3} applies.)  

For context, we compare this situation with the Szemer\'edi-Trotter theorem, the most fundamental theorem in incidence geometry.  The Szemer\'edi-Trotter theorem says that for a set of $L$ lines in $\RR^2$, the number of $r$-rich points is $\lesssim L^2 r^{-3} + L r^{-1}$.  This theorem is also true over $\CC^2$, but the proof is much harder -- cf. \cite{Toth} and \cite{Zahl}.  The situation over finite fields is not understood and is a major open problem, cf. \cite{BKT}.  In contrast, Theorem \ref{2richk3} works equally well over any field.   This makes the finite field case of Theorem \ref{2richk3} particularly interesting and useful.  For instance, Rudnev \cite{R} and Roche-Newton-Rudnev-Shkredov \cite{RRS} have applied Theorem \ref{2richk3} to prove new bounds about the sum-product problem in finite fields.  The preprint \cite{RuSe} discusses some other combinatorial problems that can be addressed using Theorem \ref{2richk3}.  

The second question we explore is the role of lines.  What happens if we replace lines by circles?  Or by other curves in $\RR^3$?  In \cite{GZ1}, Josh Zahl and I proved a version of Theorem \ref{2richR3} for algebraic curves of controlled degree.

\begin{theorem} \label{2richcurves} (\cite{GZ1}) For any $d$ there are constants $C(d), c_1(d) > 0$ so that the following holds.  Suppose that $k$ is any field.  Suppose that $\Gamma$ is a set of $L$ irreducible algebraic curves in $k^3$ of degree at most $d$.  If the characteristic of $k$ is $p > 0$, then assume in addition that $L \le c_1(d) p^2$.  Then either

\begin{itemize}

\item $|P_2(\frak L)| \le C(d) L^{3/2}$ or

\item there is an algebraic surface of degree at most $100 d^2$ that contains at least $L^{1/2}$ curves of $\frak L$.  

\end{itemize}

\end{theorem}

There are a couple reasons why I think it is natural to consider various algebraic curves instead of just straight lines.  One reason is that the proof is closely based on algebraic geometry.  Once we have a good understanding of the ideas involved, they apply naturally to all algebraic curves.  A second reason is that this more general result will probably have more applications.  For instance, we recall a little about the distinct distance problem in the plane.  In \cite{ES} Elekes and Sharir suggested an interesting new approach to the problem, connecting distinct distances in the plane to problems about the incidence geometry of some degree 2 algebraic curves in $\RR^3$.  In \cite{GK}, there is a clever change of coordinates so that these degree 2 curves become lines, and then Theorem \ref{2richR3} applies to bound the number of 2-rich points.  It appears to me that this clever change of coordinates was rather fortuitous.  I believe that most problems about algebraic curves cannot be reduced to the straight line case by a change of coordinates, and I think that when results along the lines of Theorem \ref{2richcurves} arise in applications, the curves involved will only sometimes be straight lines.    Theorem \ref{2richcurves} applies to the problem about degree 2 curves from \cite{ES}, and I think it will probably have more applications in the future.

The third question that we discuss is what happens in higher dimensions.  The situation in higher dimensions is not yet understood.  The following conjecture seems natural to me.  (Similar questions were raised in \cite{SSS} and \cite{Z}).  

\begin{conj} \label{mainconj} Let $k$ be any field.  Suppose that $\Gamma$ is a set of $L$ irreducible algebraic curves in $k^n$, of degree at most $d$.  If the characteristic of $k$ is $p> 0$, then also assume that $L \le p^{n-1}$.  

Then either

\begin{itemize}

\item $|P_2(\Gamma)| \le C(d, n) L^{\frac{n}{n-1}}$ or

\item There is a dimension $2 \le m \le n-1$, and an algebraic variety $Z$ of dimension $m$ and degree at most $D(d, n)$ so that $Z$ contains at least $L^{\frac{m-1}{n-1}}$ curves of $\Gamma$.  

\end{itemize}
\end{conj}

There is some significant progress on this conjecture in four dimensions.  In \cite{SS}, Sharir and Solomon prove estimates for $r$-rich points of a set of lines in $\RR^4$.  These estimates only apply for fairly large $r$, not $r=2$, so they don't literally address this conjecture, but they establish sharp bounds in a similar spirit for larger values of $r$.  In \cite{GZ2}, Josh Zahl and I prove a slightly weaker estimate of this form for algebraic curves in $\RR^4$.  So far nothing close to this conjecture is known for lines in $\CC^4$ or over $\FF_p^4$.  Moreover,  nothing close to this conjecture is known in higher dimensions.  I think that this is a natural question, and that if it is true, it would probably have significant applications.  If it is false, that would also be interesting, and it would point to new subtleties in incidence geometry in higher dimensions.

In the next section, we discuss the proofs of the known results.  Afterwards, we come back and discuss how much these proofs can tell us about higher dimensions, and what new issues arise.

\section{How does ruled surface theory help in the proof}

In this section, we discuss some of the ideas in the proofs of the results from the last section.  The ideas we want to discuss are easiest to explain over $\CC$, so we first state a version of Theorem \ref{2richR3} over $\CC$.

\begin{theorem} \label{2richC3}  There is a large constant $K$ so that the following holds.  Suppose that $\frak L$ is a set of $L$ lines in $\CC^3$.  Then either

\begin{itemize}

\item $|P_2(\frak L)| \le K L^{3/2}$ or

\item there is a plane or degree 2 algebraic surface that contains at least $L^{1/2}$ lines of $\frak L$.  

\end{itemize}

\end{theorem}

To get started, we think a little about the role of planes and degree 2 algebraic surfaces.  What is special about planes and degree 2 algebraic surfaces that makes them appear here?  Planes and degree 2 surfaces are doubly ruled.  A ruled surface is an algebraic surface that contains a line through every point.  A doubly ruled surface is a surface that contains two distinct lines through every point.  

At this point, we can say a little about the connection between ruled surface theory and incidence geometry.  A doubly ruled surface can be roughly thought of as an algebraic family of lines with many 2-rich points.  In incidence geometry, one tries to classify finite sets of lines with many 2-rich points.  In ruled surface theory, one tries to classify doubly ruled surfaces -- that is, algebraic families of lines with many 2-rich points.  To prove Theorem \ref{2richC3}, we begin with a finite set of lines with many 2-rich points, and we build around it a whole doubly ruled surface.  Tools from ruled surface theory help to build this surface and they help to analyze the surface once it is built, ultimately leading to information about the original finite set of lines.

Doubly ruled algebraic surfaces in $\CC^3$ were classified in the 19th century.  It turns out that planes and degree 2 surfaces are the only irreducible doubly ruled surfaces.  These surfaces appear in the statement of Theorem \ref{2richC3} because they are the only irreducible doubly ruled surfaces.  Roughly speaking, a doubly ruled surface is an algebraic family of lines with many 2-rich points.  Theorem \ref{2richC3} is telling us that a finite configuration of lines with many 2-rich points must be related to an algebraic family of lines with many 2-rich points.  

At this point, let us pause to review some vocabulary from algebraic geometry that we will use in the rest of the essay.  After we set up this vocabulary, we can state things precisely, starting with the classification of doubly ruled surfaces in $\CC^3$.  

An algebraic set in $\CC^n$ is the set of common zeroes of a finite list of polynomials in $\CC[z_1, ..., z_n]$.  An algebraic set is called reducible if it is the union of two proper algebraic subsets.  Otherwise it is called irreducible.  An irreducible algebraic set in $\CC^n$ is also called an affine variety.  

Any affine variety $V$ in $\CC^n$ has a dimension.  The dimension of $V$ is the largest number $r$ so that there is a sequence of proper inclusions of non-empty varieties $V_0 \subset ... \subset V_r = V$.  The dimension of an algebraic set in $\CC^n$ is the maximum dimension of any irreducible subset.  An algebraic curve is a variety of dimension 1.  

Using the dimension, we can define a useful notion of the generic behavior of points in a variety.  We say that a generic point of an algebraic variety $V$ obeys condition $(X)$ if the set of points $p \in V$ where $(X)$ does not hold is contained in an algebraic subset $E \subset V$ with $\dim E < \dim V$.  For instance, we say that a 2-dimensional algebraic variety $\Sigma \subset \CC^3$ is generically doubly ruled if there is a 1-dimensional algebraic set $\gamma \subset \Sigma$, and every point of $\Sigma \setminus \gamma$ is contained in two lines in $\Sigma$.

An affine variety also has a degree.  There is a non-trivial theorem which says that for any affine variety $V$ in $\CC^n$ there is unique choice of $r$ and $d$ so that a generic $(n-r)$-plane in $\CC^n$ intersects $V$ in exactly $d$ points.  The value of $r$ is the dimension of $V$, as defined above.  The value of $d$ is called the degree of $V$.  

There is a nice short summary of facts about dimension and degree in Section 4 of \cite{SoTa}, which contains everything we have mentioned here.  A fuller treatment appears in Harris's book on algebraic geometry \cite{H}.  

This is all the terminology that we will need, and we now return to discussing doubly ruled surfaces.  We can now state a classification theorem for double ruled surfaces in $\CC^3$.  

\begin{theorem} \label{classdoubruled} (Classification of doubly ruled surfaces, cf. Proposition 13.30 in \cite{G}) Suppose that $P \in \CC[z_1, z_2, z_3]$ is an irreducible polynomial and that $Z(P)$ is generically doubly ruled.  Then $P$ has degree 1 or 2, and so $Z(P)$ is a plane or a degree 2 algebraic surface.
\end{theorem}

There are three somewhat different proofs of Theorem \ref{2richC3} in the literature -- in \cite{GK}, in \cite{K}, and in \cite{GZ1}.  All three proofs use ruled surface theory in a crucial way, and this is the aspect that we will focus on.  Other parts of the argument are somewhat different in the three proofs.  The proof I want to outline here is the one from \cite{GZ1}.  Another reference is my book on polynomial methods in combinatorics, \cite{G}, which will be published in the near future by the AMS.  In the chapter on ruled surfaces in \cite{G}, I give a detailed proof of Theorem \ref{2richC3} using this method.  

For this sketch, let us suppose that each line of $\frak L$ contains about the same number of 2-rich points.  This is the most interesting case of Theorem \ref{2richC3}.  So each line contains about $K L^{1/2}$ points of $P_2(\frak L)$.  I want to highlight three stages in the proof, which I discuss in three subsections.

\subsection{Degree reduction}  \label{subsecdegred}

The first step of the argument is to find a (non-zero) polynomial $P$ that vanishes on the lines of $\frak L$ with a good bound on the degree of $P$.  For reference, given any set of $N$ points in $\CC^3$, there is a non-zero polynomial that vanishes on all these points with degree at most about $N^{1/3}$.  For a generic set of points, this bound is sharp.  By a similar argument, for any set of $L$ lines in $\CC^3$, there is a non-zero polynomial that vanishes on the lines with degree at most about $L^{1/2}$.  For a generic set of lines, this bound is also sharp.

Given that each line of $\frak L$ contains about $K L^{1/2}$ lines of $P_2(\frak L)$, we show that there is a non-zero polynomial $P$ vanishing on all the lines of $\frak L$ with degree $O (K^{-1} L^{1/2})$.  As long as $K$ is large enough, this degree is well below the degree required for a generic set of lines.  This shows that, compared to a generic set of lines, the set $\frak L$ has a little algebraic structure.  

Even though the degree of $P$ is only a little smaller than the trivial bound $L^{1/2}$, this small improvement turns out to be a crucial clue to the structure of $\frak L$, and it eventually leads to a much more precise description of $P$: $P$ is a product of irreducible polynomials of degrees 1 and 2.  Once we know this structure for the polynomial $P$, the conclusion of the theorem is easy: the lines of $\frak L$ are contained in $O(K^{-1} L^{1/2})$ planes and degree 2 algebraic surfaces.  By pigeonholing, one of these surfaces must contain at least $L^{1/2}$ lines of $\frak L$.

Here is the idea of the degree bound for $P$.  We randomly pick a subset $\frak L' \subset \frak L$ with $L' \le L$ lines, where $L'$ is a parameter that we can tune later.  Then we find a non-zero polynomial $P$ that vanishes on the lines of $\frak L'$ with degree at most $C (L')^{1/2}$.  (We will eventually choose $L'$ so that this bound is $C K^{-1} L^{1/2}$.)  If $L'$ is big enough, then with high probability the polynomial $P$ actually vanishes on all the lines of $\frak L$.  Here is the mechanism that makes this vanishing happen, which I call contagious vanishing.  By hypothesis, each line $l \in \frak L$ contains at least $K L^{1/2}$ 2-rich points of $\frak L$.  With high probability many of these points will lie in lines of $\frak L'$.  The polynomial $P$ vanishes at every point where $l$ intersects a line of $\frak L'$.  If the number of these points is more than the degree of $P$, then $P$ must vanish on the line $l$ also.  If we choose $L'$ carefully, then this mechanism will force $P$ to vanish on all the lines of $\frak L$.  Carrying out the details of this argument, the numbers work out so that the degree of $P$ is at most $C K^{-1} L^{1/2}$ -- cf. Proposition 11.5 in \cite{G}.

At this point, we factor $P$ into irreducible factors $P = \prod_j P_j$.  Each line of $\frak L$ must lie in $Z(P_j)$ for at least one $j$.  We let $\frak L_j \subset \frak L$ be the set of lines of $\frak L$ that lie in $Z(P_j)$.  We subdivide the 2-rich points as

$$P_2(\frak L) = \cup_j P_2(\frak L_j) \bigcup \textrm{``mixed 2-rich points''} ,$$

\noindent where a mixed 2-rich point is the intersection point of some line $l \in \frak L_j$ with some line $l' \notin \frak L_j$.  A line not in $\frak L_j$ can intersect $Z(P_j)$ at most $\Deg P_j$ times.  Therefore, the total number of mixed 2-rich points is at most $L (\sum_j \Deg P_j) = L \Deg P = O(K^{-1} L^{3/2})$, only a small fraction of the total number of 2-rich points.  By factoring the polynomial $P$ we have broken the original problem of understanding $\frak L$ into essentially separate subproblems of understanding each set $\frak L_j$.

The most difficult case is when $P$ is irreducible.  The general case can be reduced to this case by studying the set of lines $\frak L_j$ and the polynomial $P_j$.  From now on we assume that $P$ is irreducible.  It remains to show that $P$ has degree 1 or 2.

\subsection{Ruled surface theory}

In this subsection, we discuss some tools from ruled surface theory and how they help up to understand the polynomial $P$ in our proof sketch.  

At this point, we know that there is a polynomial $P$ that vanishes on the lines of $\frak L$ with degree significantly smaller than $L^{1/2}$, and we are focusing on the case where $P$ is irreducible.  Using this little bit of structure, we are going to find out a lot more about the polynomial $P$ and its zero set $Z(P)$.  Ultimately, we will see that $P$ has degree 1 or 2.  In this subsection, we sketch how to prove that $Z(P)$ is generically doubly ruled.  

For each 2-rich point $x \in P_2(\frak L)$, the point $x$ lies in two lines in $Z(P)$.  Since $\frak L$ has many 2-rich points, we know that there are many points in $Z(P)$ that lie in two lines in $Z(P)$ -- there are many points where $Z(P)$ ``looks doubly-ruled''.  Based on this we will show that almost every point of $Z(P)$ lies in two lines in $Z(P)$.  Loosely speaking, the property of ``looking doubly-ruled'' is contagious - it spreads from the 2-rich points of $\frak L$ and fills almost every point of $Z(P)$.   The tools to understand why this property is contagious come from ruled surface theory.

The first topic from ruled surface theory that we introduce is flecnodal points.  A point $z \in Z(P)$ is flecnodal if there is a line $l$ through $z$ which is tangent to $Z(P)$ to third order.  Here is a more formal definition, which also makes sense if $z$ is a singular point of $Z(P)$, where it's not immediately obvious what tangent to $Z(P)$ means.   A point $z \in Z(P)$ is flecnodal if there is a line $l$ with tangent vector $v$ so that

$$ 0 = P(z) = \partial_v P(z) = \partial^2_v P(z) = \partial^3_v P(z). $$

\noindent Here we write $\partial_v$ for the directional derivative in direction $v$:

$$ \partial_v := \sum_{i=1}^3 v_i \frac{\partial}{\partial z_i}, $$

\noindent and we write $\partial_v^k$ to denote repeatedly applying this differentiation -- for instance,

$$ \partial_v^2 P := \partial_v \left( \partial_v P \right). $$

If a point $z$ lies in a line in $Z(P)$, then it follows immediately that $z$ is flecnodal.  Flecnodal points are useful because they also have a more algebraic description.  A basic theme of algebraic geometry is to take any geometric property of a surface, and describe it in an algebraic way, in terms of the vanishing of some polynomials.  

\begin{theorem} \label{Sal} (Salmon, \cite{Sa}  Art. 588 pages 277-78) For any polynomial $P \in \CC[z_1, z_2, z_3]$, there is a polynomial $\Flec P \in \CC[z_1, z_2, z_3]$ so that

\begin{itemize}
\item A point $z \in Z(P)$ is flecnodal if and only if $\Flec P (z) = 0$.

\item $\Deg \Flec P \le 11 \Deg P$.  

\end{itemize}

\end{theorem}

(For some discussion of the history of this result, see the paragraph after Remark 12 in \cite{K}.)

Our goal is to connect 2-rich points and doubly-ruled surfaces, so we introduce a doubly-ruled analogue of being flecnodal.  We say that a point $z \in Z(P)$ is doubly flecnodal if there are two (distinct) lines $l_1, l_2$ through $z$, with tangent vectors $v_1, v_2$, so that for each $i=1,2$,

$$ 0 = P(z) = \partial_{v_i} P(z) = \partial^2_{v_i} P(z) = \partial^3_{v_i} P(z). $$

Doubly flecnodal points were first introduced in \cite{GZ1} and \cite{G}.  There is an analogue of Salmon's theorem for doubly flecnodal polynomials -- cf. Proposition 13.3 in \cite{G}.  It is a little more complicated to state.  Instead of one flecnodal polynomial, there is a finite list of them.  

\begin{theorem} \label{doubflecSal}
There are universal constants $J$ and $C$ so that the following holds.  For any polynomial $P \in \CC[z_1, z_2, z_3]$, there is a finite list of polynomials $\Flec_{2,j} P$, with $1 \le j \le J$, and a Boolean function $\Phi: \{0, 1\}^J \rightarrow \{0, 1\}$ so that the following holds.

\begin{itemize}

\item For each $j$, $\Deg \Flec_{2,j} P \le C \Deg P$.

\item Let $V_{2,j} P(z)$ be equal to zero if $\Flec_{2,j} P(z) = 0$ and equal to 1 otherwise.  Then $z$ is a doubly flecnodal point of $Z(P)$ if and only if

$$ \Phi \left( V_{2,1} P(z), ..., V_{2, J} P(z) \right) = 0. $$
\end{itemize}
\end{theorem}

\noindent This theorem sounds more complicated than Salmon's theorem, but in the applications we're about to describe, it is essentially equally useful.  

Because flecnodal and doubly flecnodal points have this algebraic description, they behave contagiously.  We start with the flecnodal points and then discuss the doubly flecnodal points.  We know that each line contains $K L^{1/2}$ 2-rich points of $\frak L$.  At each of these points, $\Flec P$ vanishes.  The degree of $P$ is at most $C K^{-1} L^{1/2}$, and the degree of $\Flec P$ is at most $11 \Deg P \le C' K^{-1} L^{1/2}$.  As long as we choose $K$ large enough, the number of points is larger than $\Deg \Flec P$ and it follows that $\Flec P$ vanishes along each line of $\frak L$.  Actually, since the lines of $\frak L$ are contained in $Z(P)$, we already know that every point of each line is flecnodal, but we included the last discussion as a warmup for doubly flecnodal points.  Now we know that $\Flec P$ vanishes on all $L$ lines of $\frak L$.  By a version of the Bezout theorem (cf. Theorem 6.7 in \cite{G}), $Z(P) \cap Z(\Flec P)$ can contain at most $\Deg P \cdot \Deg \Flec P$ lines, unless $P$ and $\Flec P$ have a common factor.  Because $\Deg P$ and $\Deg \Flec P$ are much less than $L^{1/2}$, we see that $P$ and $\Flec P$ must indeed have a common factor.  Since $P$ is irreducible, $P$ must divide $\Flec P$.  Therefore $\Flec P$ vanishes on $Z(P)$, and every point of $Z(P)$ is flecnodal!

Doubly flecnodal points are contagious for a similar reason.  We just do the first step of the argument.  There are $J$ polynomials $\Flec_{2,j} P$.  For each point $z$, there are $2^J$ possible values for the vector $(V_{2,1} P(z), ..., V_{2,J} P(z) )$.  Fix a line $l \in \frak L$.  By hypothesis, $l$ contains at least $K L^{1/2}$ points of $P_2(\frak L)$.  Now, by the pigeonhole principle, we can find a vector $\sigma \in \{0,1\}^{J}$ and a subset $X_\sigma \subset P_2(\frak L) \cap l$ so that 

\begin{itemize}

\item for each point $z \in X_\sigma$, $V_{2,j} P(z) = \sigma_j$. 

\item $| X_\sigma | \ge 2^{-J} K L^{1/2}$.

\end{itemize}

\noindent Because every point of $X_\sigma$ is doubly flecnodal, we see that $\Phi (\sigma) = 0$.  We choose the constant $K$ significantly larger than $2^{-J}$, and so $|X_{\sigma}| > \Deg \Flec_{2,j} P$ for each $j$.  Therefore, if $\sigma_j = 0$, then $\Flec_{2,j} P$ vanishes on the whole line $l$.  If $\sigma_j = 1$, then $\Flec_{2,j} P$ does not vanish on the whole line $l$, and so it vanishes at only finitely many points of $l$.  Therefore, for almost every $z \in l$, $\Flec_{2,j} P(z)$ vanishes if and only if $\sigma_j = 0$.  In other words, at a generic point of the line $l$, $V_{2,j} P(z) = \sigma_j$.  Therefore, at a generic point of $l$, $\Phi( V_{2,1} P(z), ..., V_{2,J} P(z) ) = \Phi (\sigma) = 0$, and so a generic point of $l$ is doubly flecnodal.  Next, by making a similar argument to the flecnodal case above, one can show that a generic point of $Z(P)$ is doubly flecnodal.  

We have now sketched the proof that $Z(P)$ is generically doubly flecnodal.  We are starting to see how the combinatorial information that $\frak L$ has many 2-rich points implies that $Z(P)$ must have a special structure.

Just because a point $z \in Z(P)$ is flecnodal, it doesn't mean that $z$ lies in a line in $Z(P)$.  For instance, let $P$ be the polynomial $P(z) = z_1^{10} + z_2^{10} + z_3^{12} - 1$ and let $z$ be the point $(1,0,0) \in Z(P)$.  If $l$ is a line through $z$ parallel to the $(z_2,z_3)$-plane, then $l$ is tangent to $Z(P)$ to ninth order.  So there are infinitely many different lines through $z$ that are tangent to $Z(P)$ to ninth order, but none of them lies in $Z(P)$.  This kind of behavior can indeed occur at some special points of $Z(P)$, but it turns out that it cannot happen at a generic point of $Z(P)$.

\begin{theorem} \label{CSM} (Cayley-Salmon-Monge) If $P \in \CC[z_1, z_2, z_3]$, and if every point of $Z(P)$ is flecnodal, then $Z(P)$ is a ruled surface -- every point of $Z(P)$ lies in a line in $Z(P)$.
\end{theorem}

(For the history of this theorem and a sketch of the proof, see the discussion around Theorem 13 in \cite{K}.)

There is also a version of this result for doubly flecnodal points (and in fact it is a little easier):

\begin{theorem} \label{CSMdoub} (cf. Proposition 13.30 in \cite{G}) If $P \in \CC[z_1, z_2, z_3]$, and if $Z(P)$ is generically doubly flecnodal, then $Z(P)$ is generically doubly ruled.
\end{theorem}

This theorem implies that our surface $Z(P)$ is generically doubly ruled.

There are several sources to read more about ruled surface theory and about the details of the arguments we have sketched here.  I tried to write readable self-contained proofs in the chapter on ruled surface theory in \cite{G}.  In Kollar's paper \cite{K}, there is a discussion of the proof of Theorem \ref{CSM} and also some history.  In Katz's ICM talk \cite{Ka}, there is another discussion of the proof of Theorem \ref{CSM}.  Finally, \cite{GZ1} gives a quite different proof of Theorem \ref{CSM} which generalizes to algebraic curves in place of straight lines.  For ruled surfaces in general, the referee suggested the classical work of Plucker \cite{P} and the modern book \cite{PW}.  

This may be a good moment to say a bit more about the theorem in \cite{GZ1}.  Suppose that $\Gamma$ is a set of $L$ circles in $\RR^3$.  For the case of circles, what kind of surfaces should play the role of planes and degree 2 surfaces?  We say that a surface $Z(P)$ is generically doubly ruled by circles if a generic point of $Z(P)$ lies in two distinct circles in $Z(P)$.  In \cite{GZ1}, it is proven that either $|P_2(\Gamma)| \le K L^{3/2}$ or $\Gamma$ contains at least $L^{1/2}$ circles in an algebraic surface $Z(P)$ which is generically doubly ruled by circles.  The same holds if circles are replaced by other classes of curves, such as parabolas, degree 3 curves, etc.  The proof follows the same outline that we have given here, and the main difficulty in the paper is to generalize the tools of ruled surface to other classes of curves.

The definition of flecnodal and doubly flecnodal involve three derivatives.  The reader may wonder why we use three derivatives.  In fact, using more than three derivatives would also work.  Using $r$ derivatives instead of three derivatives, we can define $r$-flecnodal points and doubly $r$-flecnodal points.  Theorem \ref{doubflecSal} holds for any choice of $r$ -- only the constants $C$ and $J$ depend on $r$ -- cf. Proposition 13.3 in \cite{G}.  Three derivatives is the minimum number of derivatives necessary to prove Theorem \ref{CSM} and Theorem \ref{CSMdoub}.  These theorems would be false if we assumed that only two derivatives vanish.  Here is a dimensional heuristic why three derivatives are important (suggested by the referee).  Fix a point $z$ in $Z(P)$.  In three dimensions, the space of lines through $x$ is a 2-dimensional space.  If we insist that $r$ derivatives of $P$ vanish in the tangent direction of a line, this gives us $r$ equations on the space of lines.  For $r=2$, dimensional heuristics suggest that there will typically be such a line.  But for $r = 3$, dimensional heuristics suggest that there will not be typically be such a line.  Indeed the theory of ruled surfaces shows that these heuristics are correct -- for a generic polynomial $P \in \CC[z_1, z_2, z_3]$, every point of $Z(P)$ is 2-flecnodal, but the subset of 3-flecnodal points is a lower-dimensional subvariety.

\subsection{Classification of doubly ruled surfaces}  \label{sketchclassif}

At this point in our sketch, we have shown that $Z(P)$ is generically double ruled, and we know that $P$ is irreducible.  
To finish the proof of Theorem \ref{2richC3}, we have to prove that $P$ has degree 1 or 2.  This follows from the classification of (generically) doubly ruled surfaces in Theorem \ref{classdoubruled}.

To end our sketch, we briefly discuss the proof of the classification theorem Theorem \ref{classdoubruled}.  In fact, there is a more general classification theorem for degree $d$ algebraic curves, which we discuss at the same time.

\begin{theorem} (\cite{GZ1})  Suppose that $P \in \CC[z_1, z_2, z_3]$ is an irreducible polynomial, and that $Z(P)$ is generically doubly ruled by algebraic curves of degree at most $d$.  Then $\Deg P \le 100 d^2$.  
\end{theorem}

Because a generic point of $Z(P)$ lies in two algebraic curves in $Z(P)$, it is not hard to find many algebraic curves in $Z(P)$ that intersect each other in many places.  More precisely, we can find two arbitrarily large families of curves $\gamma_{1, i}$ and $\gamma_{2,j}$ in $Z(P)$, so that for each pair $i,j$, $\gamma_{1,i}$ intersects $\gamma_{2,j}$, and all the intersection points are distinct -- cf. Lemma 11.8 in \cite{GZ1}.  The proof strongly uses the fact that $Z(P)$ is 2-dimensional.  The idea of the argument is to study the curves passing through a small ball in $Z(P)$.  For the sake of this sketch, let us suppose that each point $z \in Z(P)$ lies in exactly two algebraic curves of degree $d$, $\gamma_1(z)$ and $\gamma_2(z)$.  Let us suppose that these curves vary smoothly with $z$, and let us suppose that $\gamma_1(z)$ and $\gamma_2(z)$ intersect transversely at $z$.  (This is the moment where we use that the dimension of $Z(P)$ is 2 -- if the dimension of $Z(P)$ is greater than 2, then two curves can never intersect transversely.)  We fix a smooth point $z_0 \in Z(P)$, and then we let $z_i$ and $w_j$ be a generic sequence of points of $Z(P)$ very close to $z_0$.  The curves $\gamma_{1,i}$ and $\gamma_{2,j}$ are just $\gamma_1(z_i)$ and $\gamma_2(w_j)$.  Since $z_i$ and $w_j$ are very close to $z_0$, then $\gamma_{1,i}$ and $\gamma_{2,j}$ are small perturbations of $\gamma_1(z_0)$ and $\gamma_2(z_0)$.  Since $\gamma_1(z_0)$ and $\gamma_2(z_0)$ intersect transversely at $z_0$, then $\gamma_{1,i}$ and $\gamma_{2,j}$ must intersect at a point close to $z_0$.  

Once we have the curves $\gamma_{1,i}$ and $\gamma_{2,j}$ we can bound the degree of $P$ by using a contagious vanishing argument.  For any degree $D$, we can choose a polynomial $Q$ of degree at most $D$ that vanishes on roughly $D^2 d^{-1}$ of the curves $\gamma_{1,i}$.  On the other hand, if $\gamma_{2,j}$ does not lie in $Z(Q)$, then $Q$ can vanish on at most $d D$ points of $\gamma_{2,j}$.  We choose $D$ so that $D^2 d^{-1} \gg d D$.  Choosing $D = 100 d^2$ is big enough.  Since $Q$ vanishes on $D^2 d^{-1}$ curves $\gamma_{1, i}$, it vanishes at $D^2 d^{-1}$ points of each curve $\gamma_{2,j}$, and so it vanishes on each curve $\gamma_{2,j}$.  Now we see that $Z(Q) \cap Z(P)$ contains infinitely many algebraic curves $\gamma_{2,j}$.  By the Bezout theorem, $P$ and $Q$ must have a common factor.  Since $P$ is irreducible, $P$ must divide $Q$.  But then $\Deg P \le \Deg Q \le 100 d^2$.  

This degree reduction argument is essentially the same as the one in Subsection \ref{subsecdegred}, but we get a better bound for the degree because the curves $\gamma_{1,i}$ and $\gamma_{2,j}$ have so many 2-rich points.  Here is a big picture summary of the proof of Theorem \ref{2richC3}.  First we used the combinatorial information to prove that the set of lines $\frak L$ has a little algebraic structure -- the lines lie in $Z(P)$ where the degree of $P$ is a bit smaller than for generic lines.  If $P$ is reducible, we divide the problem into essentially disjoint subproblems, and we assume from now on that $P$ is irreducible.  Second, we use the degree bound on $P$ and the combinatorial information about the lines to prove that $Z(P)$ is generically doubly ruled.  So our finite set of lines $\frak L$ fits into an algebraic family of lines with many 2-rich points.  Third, we extend $\frak L$ by adding a lot of other lines from the surface $Z(P)$.  By doing this, we can amplify the number of 2-rich points.  We get a new set of $N \gg L$ lines in $Z(P)$ with around $N^2$ 2-rich points.  Finally, we apply degree reduction to this bigger set of lines, and we get a much stronger estimate for the degree of $P$.

\section{Thoughts about higher dimensions}

In this last section, we reflect on how much ruled surface theory can tell us about incidence geometry in higher dimensions, and we point out some open problems.  What happens if we try to adapt the proof of Theorem \ref{2richC3} that we just sketched to higher dimensions?  We broke the proof of Theorem \ref{2richC3} into three stages.  We discuss each stage, but especially focusing on the last stage -- the classification of doubly ruled surfaces.

We suppose that $\frak L$ is a set of $L$ lines in $\CC^n$.  We suppose that $|P_2(\frak L)| \ge K L^{\frac{n}{n-1}}$.  We also make the minor assumption that each line contains about the same number of 2-rich points: so each line contains at least $K L^{\frac{1}{n-1}}$ points of $P_2(\frak L)$.  

\subsection{Degree reduction}  The degree reduction stage works in any dimension.  In $n$ dimensions, for any set of $N$ points, there is a polynomial of degree at most $C_n N^{1/n}$ vanishing on the set, and this bound is sharp for generic sets.  For any set of $L$ lines, there is a polynomial of degree at most $C_n L^{\frac{1}{n-1}}$ vanishing on each line, and this bound is sharp for generic sets of lines.  But if each line of $\frak L$ contains at least $K L^{\frac{1}{n-1}}$ 2-rich points of $\frak L$, then there is a polynomial $P$ vanishing on the lines of $\frak L$ with degree at most $C_n K^{\frac{-1}{n-2}} L^{\frac{1}{n-1}}$.  So we see that in any number of dimensions, if $K$ is large enough then $\frak L$ has some algebraic structure.  I think this suggests that it is a promising avenue to try to study $\frak L$ using algebraic geometry.

\subsection{Ruled surface theory}  Some of the tools we used in the second stage have generalizations to higher dimensions.  Landsberg \cite{L} has proven a version of Theorem \ref{CSM} in any number of dimensions.  Sharir and Solomon \cite{SS} generalized the flecnode polynomial to four dimensions and proved the four-dimensional analogue of Theorem \ref{Sal}.  Double-flecnode polynomials have so far only been defined in three dimensions.  In higher dimensions, there is one technical point which will be more difficult.  In $\CC^n$, there are doubly ruled varieties of every dimension between 2 and $n-1$.  Therefore, it is not enough to consider algebraic hypersurfaces, which can be written in the form $Z(P)$ for a single polynomial $P$ - we have to consider algebraic varieties of all dimensions.  If one could generalize the methods in this second stage to higher dimensions, it might be possible to prove the following conjecture.

\begin{conj} \label{conjdoubruled} Suppose that $\frak L$ is a set of $L$ lines in $\CC^n$.  Then either

\begin{itemize}

\item $|P_2(\frak L)| \le C(n) L^{\frac{n}{n-1}}$ or

\item There is a dimension $2 \le m \le n-1$, and a generically double-ruled affine variety $Z$ of dimension $m$ so that $Z$ contains at least $L^{\frac{m-1}{n-1}}$ lines of $\frak L$.  (Recall that an affine variety is irreducible by definition.)

\end{itemize}
\end{conj}

We can generalize this conjecture to algebraic curves as follows.

\begin{conj} \label{conjdoubruledcurv} Suppose that $\Gamma$ is a set of $L$ irreducible algebraic curves in $\CC^n$ of degree at most $d$.  Then either

\begin{itemize}

\item $|P_2(\Gamma)| \le C(d, n) L^{\frac{n}{n-1}}$ or

\item There is a dimension $2 \le m \le n-1$, and an (irreducible) affine variety $Z$ of dimension $m$ which is generically doubly ruled by algebraic curves of degree at most $d$ and contains at least $L^{\frac{m-1}{n-1}}$ curves of $\Gamma$.

\end{itemize}
\end{conj}

If Conjecture \ref{conjdoubruled} and/or \ref{conjdoubruledcurv} is true, it would point to a strong connection between incidence geometry and ruled surface theory.  On the other hand, it would probably not be useful in applications unless we could also prove a classification of doubly ruled varieties -- at least a very rough classification.  So let us turn now to the problem of the classification of doubly ruled varieties.

\subsection{Classification of doubly ruled varieties} The classification of doubly ruled surfaces in $\CC^3$ was fairly simple, but in higher dimensions, this part of the problem may become a lot more complex.  I would like to propose a question about doubly ruled varieties that could be useful to understand for applications to incidence geometry.  

To get started, we might ask, if $Y^m \subset \CC^n$ is a generically doubly ruled (irreducible) variety, does it follow that $\Deg Y \le C(n)$?  The answer to this question is no.  It may happen that every point of $Y$ lies in a 2-plane in $Y$.  Such a variety is clearly doubly ruled, and it may have an arbitrarily high degree.  For a high degree example, suppose that $Y$ is a graph of the form

$$ z_4 = P_1(z_3) z_1 + P_2(z_3) z_2 + Q(z_3), $$

\noindent where $P_1, P_2,$ and $Q$ are polynomials of high degree.  If $w = (w_1, w_2, w_3, w_4) \in Y$, then $w$ lies in the following 2-plane in $Y$:

$$ z_3 = w_3; z_4 = P_1(w_3) z_1 + P_2(w_3) z_2 + Q(w_3). $$

\noindent If $P_1, P_2$, or $Q$ have high degree, then $Y$ will have high degree also.  (Also the algebraic set $Y$ is in fact irreducible for any chocie of $P_1, P_2, Q$.)

Suppose for a moment that the variety $Y$ that we find in the second stage is a graph of this form, and suppose for simplicity that every line of $\frak L$ lies in $Y$.  For a typical $P_1, P_2, Q$, every line in $Y$ is contained in one of the planes above.  Suppose for a moment that our variety $Y$ has this convenient property.  Then we can separate the lines of $\frak L$ into subsets corresponding to different 2-planes.  Since each line of $\frak L$ contains at least $K L^{\frac{1}{n-1}}$ 2-rich points, one of the 2-planes must contain at least $K L^{\frac{1}{n-1}}$ lines of $\frak L$, and this satisfies the conclusion of Conjecture \ref{mainconj}.

I don't know whether there are more exotic examples of doubly ruled varieties than this one.  Let me introduce a little vocabulary so that I can make an exact question.  We say that a variety $Y$ is ruled by varieties with some property $(*)$ if each point $y \in Y$, lies in a variety $X \subset Y$ where $X$ has property $(*)$.  We say that a variety $Y$ is doubly ruled by varieties with property $(*)$ if each point $y$ lies in two distinct varieties $X_1, X_2 \subset Y$ with property $(*)$.  We say that $Y$ is generically ruled by varieties with property $(*)$ if a generic point $y \in Y$ lies in a variety $X \subset Y$ with property $(*)$, and so on.

\begin{question} \label{classques} Suppose that $Y \subset \CC^n$ is a variety which is generically doubly ruled (by lines).  Does it follow that $Y$ is generically ruled by varieties with dimension at least 2 and degree at most $C(n)$?
\end{question}

To the best of my knowledge this question is open.  Noam Solomon pointed me to a relevant paper in the algebraic geometry literature by Mezzetti and Portelli, \cite{MP}.  Under a technical condition, this paper gives a classification of doubly ruled 3-dimensional varieties in $\mathbb{CP}^4$ -- see Theorem 0.1.  The technical condition is that the Fano scheme of lines of $Y$ is generically reduced.  If $Y$ is generically doubly ruled and obeys this condition, then the classification from Theorem 0.1 of \cite{MP} implies that either $Y$ has degree at most 16 or $Y$ is generically ruled by 2-dimensional varieties of degree at most 2.  

We can also pose more general questions in a similar spirit to Question \ref{classques}.  

\begin{question} \label{classquescurv} Suppose that $Y \subset \CC^n$ is generically doubly ruled by (irreducible) algebraic curves of degree at most $d$.  Does it follow that $Y$ is generically ruled by varieties with dimension at least 2 and degree at most $C(d,n)$?
\end{question}

\begin{question} \label{classquesm} Suppose that $Y \subset \CC^n$ is generically doubly ruled by varieties of dimension $m$ and degree at most $d$.  Does it follow that $Y$ is generically ruled by varieties with dimension at least $m+1$ and degree at most $C(d, m, n)$?
\end{question}

If the answers to Questions \ref{classques} and \ref{classquescurv} are affirmative, then I think it would be promising to try to prove Conjecture \ref{mainconj} using tools from ruled surface theory.  If the answer to Question \ref{classques} is no, then it means that there are some exotic doubly ruled varieties $Y \subset \CC^n$.  These varieties would be a potential source of new examples in incidence geometry, and could possibly lead to counterexamples to Conjecture \ref{mainconj}.  

For a given variety $Y$ containing many lines, it looks interesting to explore incidence geometry questions for sets of lines in $Y$.  This circle of questions was raised by Sharir and Solomon in \cite{SS}.  In particular, they raised the following question.

\begin{question} \label{rrichinquad} Suppose that $Y$ is the degree 2 hypersurface in $\RR^4$ defined by the equation

$$ x_1 = x_2^2 + x_3^2 - x_4^2. $$

\noindent For a given $r$, what is the maximum possible size of $|P_r(\frak L)|$? \end{question}

This question was studied by Solomon and Zhang in \cite{SZ}, building on earlier work of Zhang \cite{Z}.  They constructed an example with many $r$-rich points.  Counting the number of $r$-rich points in the example is non-trivial and they used tools from analytic number theory to do so.  Their construction gives $\sim L^{3/2} r^{-3}$ $r$-rich points.  Since a generic point of $Y$ lies in infinitely many lines in $Y$, it is easy to produce examples with $\sim L r^{-1}$ $r$-rich points, so their example is interesting when $r$ is smaller than $L^{1/4}$.  The best known upper bound on $|P_r(\frak L)|$ is based on a random projection argument.  Rudnev used a closely related random projection argument in \cite{R} -- cf. the bottom of page 6 of \cite{R}.  We note that $Y$ does not contain any 2-plane.  Since $Y$ is the zero set of a degree 2 polynomial, the intersection of $Y$ with a 2-plane may contain at most two lines.  Therefore, we see that $\frak L$ contains at most two lines in any 2-plane.  Now we let $\frak L'$ be the projection of $\frak L$ to a generic 3-plane.  For a generic choice of the projection we see that $| \frak L' | = | \frak L|$, $|P_r(\frak L')| = |P_r(\frak L)|$, and $\frak L'$ contains at most two lines in any 2-plane.  We then bound $|P_r(\frak L')|$ using Theorem 4.5 from \cite{GK}, giving the bound $|P_r(\frak L)| = |P_r(\frak L')| \lesssim L^{3/2} r^{-2} + L r^{-1}$.  There is a large gap between the upper and lower bounds.  The random projection argument does not seem to use much of the structure of $Y$: as Rudnev points out in \cite{R}, the space of lines in $\RR^3$ is 4-dimensional while the space of lines in $Y$ is only 3-dimensional.  

We can ask the same question over the complex numbers.  The example of \cite{SZ} is still the best lower bound.  For upper bounds, the random projection argument still works, but Theorem 4.5 from \cite{GK} is not known over the complex numbers.  In the complex case, the best upper bound comes from applying Theorem 2 of \cite{K}, giving the bound $|P_r(\frak L)| \lesssim L^{3/2} r^{-3/2} + L r^{-1}$.  

In Question \ref{rrichinquad}, the interesting case is when $r > 2$.  The variety $Y$ contains the subvariety $x_1 = x_3^2 - x_4^2$, $x_2 = 0$.  It is not difficult to construct a set of $L$ lines in this subvariety with $L^2/4$ 2-rich points by modifying the example at the start of Section 2.  But for cubic hypersurfaces, it looks hard to estimate the number of 2-rich points.  For example, we can ask the following question.

\begin{question} \label{2richincubic} Suppose that $Y$ is the degree 3 hypersurface in $\CC^4$ defined by the equation

$$ z_1^3 + z_2^3 + z_3^3 + z_4^3 = 1. $$

Suppose that $\frak L$ is a set of $L$ lines in $Y$.  What is the maximum possible size of $P_2(\frak L)$?

\end{question}

(I believe that a generic point of this cubic hypersurface $Y$ lies in six lines in $Y$.  Here is a heuristic argument for this guess.  Fix a point $p \in Y$ and translate the coordinate system so that $p = 0$.  In the new coordinate system, $Y$ is given as the zero set of a polynomial $P$ of the form $P = P_3(z) + P_2(z) + P_1(z)$, where $P_i(z)$ is homogeneous of degree $i$.  (There is no zeroth order term because we have arranged that $0 \in Z(P)$ and so $P(0) = 0$.)   For a non-zero $z$, the line from 0 through $z$ lies in $Y = Z(P)$ if and only if $P_3(z) = P_2(z) = P_1(z)=0$.  So the set of lines in $Y$ through $p$ is given by intersecting a degree 3 hypersurface, a degree 2 hypersurface, and a degree 1 hypersurface in $\mathbb{CP}^3$.  For a generic choice of these hypersurfaces, the intersection will consist of six elements of $\mathbb{CP}^3$, and I believe that this occurs at a generic point of $Y$.)

Note that it does matter which cubic hypersurface we consider.  The cubic hypersurface $z_4 = z_1 z_2 z_3$ contains a 2-dimensional degree 2 surface defined by $z_3 = 1$, $z_4 = z_1 z_2$, and this surface contains $L$ lines with $L^2/4$ 2-rich points, as in the example at the start of Section 2.  I believe that the cubic hypersurface $z_1^3 + z_2^3 + z_3^3 + z_4^3 = 1$ does not contain any 2-dimensional variety of degree 2.  If this is the case, then we can get a non-trivial upper bound by a random projection argument.  By a version of the Bezout theorem, the intersection of $Y$ with any degree 2 2-dimensional variety will contain at most 6 lines.  Randomly projecting $\frak L$ to $\CC^3$, we get a set of lines $\frak L'$ with at most 6 lines of $\frak L'$ in any 2-plane or degree 2 surface.  Then applying Theorem \ref{2richC3}, we see that $|P_2(\frak L)| = |P_2(\frak L')| \lesssim L^{3/2}$.  But I suspect that the maximum size of $|P_2(\frak L)|$ is much smaller than $L^{3/2}$.  

I think that these questions about lines in low degree varieties are a natural direction of research in incidence geometry.  If there are more exotic doubly-ruled varieties $Y$, then it would also be natural to study analogous questions for them.

\end{document}